\documentclass[11pt]{article}

\usepackage{amsmath,amsthm,verbatim,amssymb,amsfonts,amscd,graphicx}
\usepackage{hyperref}
\hypersetup{colorlinks,linkcolor={black},citecolor={black},urlcolor={black}}
\usepackage{bm}
\usepackage{tikz}
\usetikzlibrary{arrows,shapes,positioning}
\usetikzlibrary{decorations.markings}
\tikzstyle arrowstyle=[scale=1]
\tikzstyle directed=[postaction={decorate,decoration={markings,mark=at position 0.75 with{\arrow[arrowstyle]{stealth}}}}]
\newtheorem{theorem}{Theorem}[section]
\newtheorem{corollary}[theorem]{Corollary}
\newtheorem{definition}[theorem]{Definition}

\newtheorem{lemma}[theorem]{Lemma}

\newtheorem{claim}{Claim}
\newenvironment{Proof}{\noindent {\bf Proof.}}{\rule{3mm}{3mm}\par\medskip}

\DeclareMathOperator{\supp}{{\rm supp}}
\DeclareMathOperator{\suppc}{{\rm supp}^{\mathrm c}}

\newcommand{\JCTB}{{\it J. Combin. Theory Ser. B}}

\newcommand{\JGT}{{\it J. Graph Theory}}

\newcommand{\DM}{{\it Discrete Math.}}

\newcommand{\SIAMDM}{{\it SIAM J. Discrete Math.}}

\newcommand{\CJM}{{\it Canad. J. Math.}}
\newcommand{\JLMS}{{\it J. London Math. Soc.}}
\newcommand{\PLMS}{{\it Proc. London Math. Soc.}}

\newcommand{\EJC}{{\it European J. Combin.}}

\newcommand{\ElecCom}{{\it Electronic Journal of Combinatorics}}

\begin{document}
\title{Signed graphs: from modulo flows to integer-valued flows}
\author{Jian Cheng, You Lu, Rong Luo, Cun-Quan Zhang\\
Department of Mathematics\\
West Virginia University\\
Morgantown, WV 26506\\
Email: \{jiancheng, yolu1, rluo, cqzhang\}@math.wvu.edu }
\date{}
\maketitle

\begin{abstract}
Converting modulo flows into integer-valued flows is one of the most critical steps in the study of integer flows.
Tutte and Jaeger's pioneering work shows the equivalence of modulo flows and integer-valued flows for ordinary graphs.
However, such equivalence does not hold any more for signed graphs.
This motivates us to study how to convert modulo flows into integer-valued flows for signed graphs.
In this paper, we generalize some early results by Xu and Zhang (\DM~299, 2005),
Schubert and Steffen (\EJC~48, 2015), and Zhu (\JCTB~112, 2015),
and show that, for signed graphs, every modulo $(2+\frac{1}{p})$-flow with $p \in {\mathbb Z}^+ \cup \{\infty\}$
can be converted/extended into an integer-valued flow.
\end{abstract}

\begin{center}
{\bf\small Keyworks:} {\small\it Signed graph; Integer flow; Circular flow; Modulo orientation}
\end{center}

\section{Introduction}

In flow theory, an integer-valued flow and a modulo flow are different by their definitions.
For ordinary graphs, Tutte showed that {\em a graph admits an integer-valued nowhere-zero $k$-flow
if and only if it admits a modulo nowhere-zero $k$-flow.}
We also notice that although most landmark results are stated as integer-valued flow results,
due to the theorem by Tutte, they were initially proved for modulo flows, such as,
the $8$-flow theorem by Jaeger~\cite{Jaeger1979},
the $6$-flow theorem by Seymour~\cite{Seymour1981},
and the weak $3$-flow theorem by Thomassen~\cite{Thomassen2012}.

However, Tutte's result cannot be applied for signed graphs (see Fig.~\ref{FIG: no-3-flow}).
That is, there is a big gap between modulo flows and integer-valued flows for signed graphs.
The first known result was proved by Bouchet~\cite{Bouchet1983} in his study of chain-groups.
\begin{theorem}
[\cite{Bouchet1983}, Proposition 3.5]\label{TH: Bouchet}
If a signed graph $(G,\sigma)$ admits a modulo $k$-flow $f_1$,
then it admits an integer-valued $2k$-flow $f_2$ with $\supp(f_1)\subseteq \supp(f_2)$.
\end{theorem}

In this paper, Theorem~\ref{TH: Bouchet} is improved for some important cases:
modulo $2$-flows, modulo $3$-flows, and modulo circular $(2+\frac{1}{p})$-flows.

\subsection{Basic definitions}

Graphs considered here may have multiple edges or loops.
Let $G$ be a graph with {\em vertex set} $V(G)$ and {\em edge set} $E(G)$.
For a vertex $v$, we denote by $E_G(v)$ the set of edges incident with $v$,
and denote $d_G(v) = |E_G(v)|$ (known as the {\em degree} of $v$).
When no confusion is caused, we simply use $E(v)$ and $d(v)$ for short. Let $X$ and $Y$ be two disjoint vertex sets.
We denote by $E(X,Y)$ the set of edges with one end in $X$ and the other end in $Y$, and by $e(X,Y) = |E(X,Y)|$.
An edge set $F$ is an {\em odd-$\lambda$-edge cut} if $|F|=\lambda$ is odd and $G-F$ has more components than $G$.
A graph $G$ is {\em odd-$\lambda$-edge-connected} if it contains no odd-$k$-edge cut for any $k \leq \lambda -2$.
The {\em odd-edge-connectivity} of $G$ is the smallest integer $\lambda$ for which $G$ is odd-$\lambda$-edge-connected.
If $F=\{e\}$, then $e$ is a {\em bridge} of $G$. A graph $G$ is {\em bridgeless} if it contains no bridges.

A {\em signed graph} is a graph $G$ associated with a {\em signature} $\sigma\colon E(G)\to\{\pm1\}$.
An edge $e$ is {\em positive} if $\sigma(e) =1$ and {\em negative} otherwise.
Every edge of $G$ consists of two half-edges, each of which is incident with exactly one end of this edge.
For a vertex $v$, denote by $H(v)$ the set of all half-edges incident with $v$.
Let $H(G) =\bigcup_{v\in V(G)} H(v)$.
For a half-edge $h$, we use $e_{h}$ to denote the edge containing $h$.
An {\em orientation} of $(G,\sigma)$ is a mapping $\tau\colon H(G)\to\{\pm1\}$
such that $\tau(h_{1})\tau(h_{2}) = -\sigma(e)$ for $e\in E(G)$,
where $h_{1}$ and $h_{2}$ are the two half-edges of $e$.

For a signed graph $(G,\sigma)$, {\it switching} at a vertex $u$ means reversing the signs of all edges incident with $u$.
Let $\mathcal{X}_{(G,\sigma)}$ be the set of signatures of $G$ obtained from $\sigma$ via a sequence of switching operations.
The {\em negativeness} of $G$ is the smallest integer $q$ for which $G$ has a signature $\sigma'\in\mathcal{X}_{(G,\sigma)}$
with exactly $q$ negative edges.

\subsection{Integer-valued flows in signed graphs}

\begin{definition}
Let $(G,\sigma)$ be a signed graph associated with an orientation $\tau$.
Let $k$ be a positive integer and $f\colon E(G) \to \mathbb Z$ be a mapping such that
$0\leq|f(e)|\leq(k-1)$ for every edge $e\in E(G)$.
The {\em boundary} of $f$ at a vertex $v$ is defined as $\partial f(v) =\sum_{h\in H(v)} f(e_{h})\tau (h)$.
The mapping $f$ is an {\em integer-valued $k$-flow} (resp.~{\em modulo $k$-flow}) of $(G,\sigma)$
if $\partial f(v) = 0$ (resp.~$\partial f(v)\equiv 0 \pmod k$) for each vertex $v \in V(G)$.
\end{definition}

Let $f$ be a flow of a signed graph $(G,\sigma)$.
The {\em support} of $f$, denoted by $\supp(f)$, is the set of edges $e$ with $f(e)\neq 0$.
A flow $f$ is {\em nowhere-zero} if $\supp(f) = E(G)$.
For convenience, we respectively shorten the notations of nowhere-zero $k$-flows into
integer-valued $k$-NZFs and modulo $k$-NZFs.

To verify Bouchet's $6$-flow conjecture~\cite{Bouchet1983} for $6$-edge-connected signed graphs,
Xu and Zhang~\cite{Xu2005} proved the following two results, which generalize Tutte's theorem to signed graph with $k=2,3$.

\begin{theorem}
[\cite{Xu2005}]\label{TH: Xu-Zhang-1}
If a signed graph $(G,\sigma)$ admits a modulo $2$-flow $f_1$ such that
each component of $\supp(f_1)$ contains an even number of negative edges,
then it also admits an integer-valued $2$-flow $f_2$ with $\supp(f_1)=\supp(f_2)$.
\end{theorem}

\begin{theorem}
[\cite{Xu2005}]\label{TH: Xu-Zhang-2}
If a signed graph $(G,\sigma)$ admits a modulo $3$-flow $f_1$
such that $\supp(f_1)$ is bridgeless,
then it also admits an integer-valued $3$-flow $f_2$
with $\supp(f_1)=\supp(f_2)$.
\end{theorem}

In this paper, under the weaker conditions, we prove the following two results
which are analogs of Theorem~\ref{TH: Bouchet} and 
respectively improve Theorems~\ref{TH: Xu-Zhang-1} and~\ref{TH: Xu-Zhang-2}.

\begin{theorem}
\label{TH: 2-to-3}
If a signed graph $(G,\sigma)$ is connected and admits a modulo $2$-flow $f_1$
such that $\supp(f_1)$ contains an even number of negative edges,
then it also admits an integer-valued $3$-flow $f_2$
with $\supp(f_1)=\{e\in E(G)\colon f_2(e)=\pm 1\}$.
\end{theorem}

\begin{theorem}
\label{TH: 3-to-4}
If a signed graph $(G,\sigma)$ is bridgeless and admits a modulo $3$-flow $f_1$,
then it also admits an integer-valued $4$-flow $f_2$ with
$\supp(f_1)\subseteq\{e\in E(G)\colon f_2(e)=\pm 1,\pm 2\}$.
\end{theorem}

\begin{figure}
\begin{center}
\begin{tikzpicture}[scale=1.2]
\path(-20:2.66) coordinate (u1);\draw [fill=black] (u1) circle (0.08cm);
\path (00:1.50) coordinate (u2);\draw [fill=black] (u2) circle (0.08cm);
\path (20:2.66) coordinate (u3);\draw [fill=black] (u3) circle (0.08cm);
\path (00:2.50) coordinate (u4);

\draw [directed, line width=0.85] (u2)--(u1);
\draw [directed, line width=0.85] (u2)--(u3);
\draw [directed, line width=0.85] (u4)--(u1);
\draw [directed, line width=0.85] (u4)--(u3);

\draw [directed, line width=0.85] plot [smooth] coordinates {(-20:2.66) (0:3.2) (20:2.66)};
\draw [directed, line width=0.85] plot [smooth] coordinates {(20:2.66) (0:3.2) (-20:2.66)};
\path (200:2.66) coordinate (v1);\draw [fill=black] (v1) circle (0.08cm);
\path (180:1.50) coordinate (v2);\draw [fill=black] (v2) circle (0.08cm);
\path (160:2.66) coordinate (v3);\draw [fill=black] (v3) circle (0.08cm);
\path (180:2.50) coordinate (v4);

\draw [directed, line width=0.85] (v2)--(v1);
\draw [directed, line width=0.85] (v2)--(v3);
\draw [directed, line width=0.85] (v4)--(v1);
\draw [directed, line width=0.85] (v4)--(v3);

\draw [directed, line width=0.85] plot [smooth] coordinates {(200:2.66) (180:3.2) (160:2.66)};
\draw [directed, line width=0.85] plot [smooth] coordinates {(160:2.66) (180:3.2) (200:2.66)};
\path (250:2.66) coordinate (w1);\draw [fill=black] (w1) circle (0.08cm);
\path (270:1.50) coordinate (w2);\draw [fill=black] (w2) circle (0.08cm);
\path (290:2.66) coordinate (w3);\draw [fill=black] (w3) circle (0.08cm);
\path (270:2.50) coordinate (w4);

\draw [directed, line width=0.85] (w2)--(w1);
\draw [directed, line width=0.85] (w2)--(w3);
\draw [directed, line width=0.85] (w4)--(w1);
\draw [directed, line width=0.85] (w4)--(w3);

\draw [directed, line width=0.85] plot [smooth] coordinates {(250:2.66) (270:3.2) (290:2.66)};
\draw [directed, line width=0.85] plot [smooth] coordinates {(290:2.66) (270:3.2) (250:2.66)};
\path (0:0) coordinate (o);\draw [fill=black] (o) circle (0.08cm);
\draw [directed, line width=0.85] (u2)--(o);
\draw [directed, line width=0.85] (v2)--(o);
\draw [directed, line width=0.85] (w2)--(o);
\end{tikzpicture}
\end{center}
\caption{\small$(G,\sigma)$ admits a modulo $3$-NZF with all edges assigned with $1$, but no integer-valued $3$-NZF.}
\label{FIG: no-3-flow}
\end{figure}
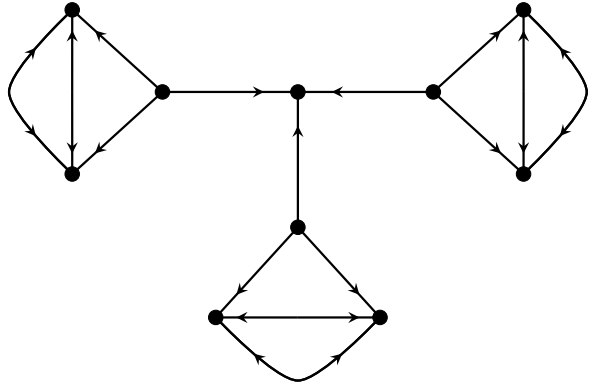

\subsection{Integer-valued circular flows in signed graphs}

\begin{definition}
Let $(G,\sigma)$ be a signed graph associated with an orientation $\tau$.
\begin{enumerate}
\item
Let $k$ and $d$ be two positive integers.
An {\em integer-valued (resp.~modulo) circular $\frac{k}{d}$-flow} of $(G,\sigma)$
is an integer-valued (resp.~modulo) flow $f$
such that $d\leq |f(e)|\leq k-d$ for every edge $e\in E(G)$.
\item
Let $p$ be a positive integer.
The orientation $\tau$ is a {\em modulo $(2p+1)$-orientation} if
$\sum_{e\in H(v)}\tau(e)\equiv0\pmod{2p+1}$ for every vertex $v\in V(G)$.
\end{enumerate}
\end{definition}

When $k=3$, Tutte's theorem~\cite{Tutte1949} implies that
a graph $G$ admits a modulo circular $3$-flow if and only if it admits an integer-valued circular $3$-flow.
This result was generalized to integer-valued circular $(2+\frac{1}{p})$-flows by Jaeger~\cite{Jaeger1981} as follows.

\begin{theorem}
[\cite{Jaeger1981}]\label{TH: Jaeger}
Let $G$ be a graph. Then the following statements are equivalent:
\begin{enumerate}
\item [\rm\bf(A)]
$G$ admits a modulo $(2p+1)$-orientation.
\item [\rm\bf(B)]
$G$ admits a modulo circular $(2+\frac{1}{p})$-flow.
\item [\rm\bf(C)]
$G$ admits an integer-valued circular $(2+\frac{1}{p})$-flow.
\end{enumerate}
\end{theorem}

For signed graphs, using an identical proof in~\cite{Jaeger1981},
one can easily prove that {\bf(A)} and {\bf(B)} are still equivalent.
However, similar to the argument for modulo flows, the equivalence relation between {\bf(B)} and {\bf(C)}
does not hold for signed graphs (see Fig.~\ref{FIG: no-3-flow}).
For more details, readers are referred to~\cite{Kaiser2014},~\cite{Macajova2014},~\cite{Raspaud2011},
\cite{Schubert2015},~\cite{Xu2005},~\cite{Zhu2015}, etc.

\medskip
The following are some early results proved by Xu and Zhang~\cite{Xu2005},
Schubert and Steffen~\cite{Schubert2015}, and Zhu~\cite{Zhu2015}.

\begin{theorem}
\label{TH: previous-results}
Let $(G,\sigma)$ be a signed graph. Then {\bf(B)} and {\bf(C)} are equivalent if
\begin{enumerate}

\item {\rm(\cite{Xu2005})}
$p=1$, and, $(G,\sigma)$ is cubic and contains a perfect matching;

\item {\rm(\cite{Schubert2015})}
$(G,\sigma)$ is $(2p+1)$-regular and contains an $p$-factor;

\item {\rm(\cite{Zhu2015})}
$(G,\sigma)$ is $(12p-1)$-edge-connected with negativeness even or at least $(2p+1)$.
\end{enumerate}
\end{theorem}

In this paper, we improve all the results in Theorem~\ref{TH: previous-results} as follows.

\begin{theorem}
\label{TH: mod-circular}
{\bf(B)} and {\bf(C)} are equivalent for signed graphs with odd-edge-connectivity at least $(2p+1)$.
That is, if a signed graph $(G,\sigma)$ is odd-$(2p+1)$-connected,
then it admits a modulo circular $(2+\frac{1}{p})$-flow
if and only if it admits an integer-valued circular $(2+\frac{1}{p})$-flow.
\end{theorem}

\section{Proof of Theorem~\ref{TH: 2-to-3}}
\label{SEC: 2-to-3}

Let $(G,\sigma)$ together with a flow $f_{1}$ be a counterexample to Theorem~\ref{TH: 2-to-3} such that $|E(G)|$ is minimized.
In the following context, we are to yield a contradiction by showing that $(G,\sigma)$ actually 
admits an integer-valued $3$-flow $f_{2}$ satisfying Theorem~\ref{TH: 2-to-3}. 
For convenience, denote $B =\supp(f_{1})$.

\begin{claim}
\label{CL: cut-edge}
$B\neq E(G)$ and each edge of $E(G)-B$ is a bridge.
\end{claim}
\begin{proof}
If $B=E(G)$, then $G$ is an eulerian graph containing an even number of negative edges.
By Theorem~\ref{TH: Xu-Zhang-1}, $G$ admits an integer-valued $2$-NZF $f_{2}$.
If $e^{*}\in E(G)-B$ is not a bridge, let $G' = G-\{e^{*}\}$.
Then $G'$ is connected and $f_{1}$ is a modulo $2$-flow of $G'$ with $|E(G')|<|E(G)|$.
Thus by the minimality of $(G,\sigma)$, $(G',\sigma)$ admits an integer-valued $3$-flow $f_{2}$
with $B =\{e\in E(G')\colon f_{2}(e)=\pm1\}$.
In both cases, $f_{2}$ is a desired integer-valued $3$-flow.
\end{proof}

\begin{claim}
\label{CL: odd-odd}
For an edge $e\in E(G)-B$, denote the components of $G-\{e\}$ by $Q_{1}$ and $Q_{2}$.
Then each $B\cap Q_{i}$ contains an odd number of negative edges.
\end{claim}
\begin{proof}
Since $B$ contains an even number of negative edges,
$B\cap Q_{1}$ and $B\cap Q_{2}$ contain the same parity number of negative edges.
Suppose to the contrary that each contains an even number of negative edges.
For $i\in\{1,2\}$, we have $|E(Q_{i})| < |E(G)|$ and therefore
$(Q_{i},\sigma)$ admits an integer-valued $3$-flow $g_{i}$
such that $B\cap Q_{i} =\{e\in E(Q_{i})\colon g_{i}(e)=\pm1\}$.
We define $f_2$ as $f_2(e') = g_i(e')$ for each $e' \in Q_i$ and $f_2(e) = 0$.
It is easy to see that $f_{2}$ is a desired integer-valued $3$-flow.
\end{proof}

Now we first choose an edge $e^{*}$ in $E(G)-B$ and denote its ends by $x_{1}$ and $x_{2}$, respectively.
For each $i\in\{1,2\}$, let $Q_{i}$ be the component of $G-\{e^{*}\}$ with $x_i\in V(Q_{i})$.
We construct a new signed graph $(H_{i},\sigma_{i})$ from $Q_{i}$ by adding a negative loop $e_{i}$ at $x_{i}$.
Denote $B_{i} = (B\cap Q_i)\cup\{e_{i}\}$ and assign $f_{1}(e_{i}) = 1$.
By Claim~\ref{CL: odd-odd}, each $B_{i}$ contains an even number of negative edges.
Therefore, $f_1$ is a modulo $2$-flow of $(H_i,\sigma_{i})$ with support $B_{i}$.
Since $|E(H_{i})| < |E(G)|$, by the minimality of $G$, 
$(H_i,\sigma_{i})$ admits an integer-valued $3$-flow $g_{i}$ such that $B_{i} =\{e\in E(H_{i})\colon g_{i}(e)=\pm1\}$.
Note that $|\partial g_{i}(x_{i})|=2$ in $Q_{i}$.
Without loss of generality, we can assume that $\partial g_{2}(x_{2}) = -\sigma(e^{*})\partial g_{1}(x_{1})$
otherwise we can replace $g_{1}$ by $-g_{1}$.
Finally, we define $f_{2}$ by assigning $f_{2}(e) = g_{i}(e)$ for each $e\in E(Q_{i})$, 
and by choosing $f_{2}(e^{*}) = 2$ or $-2$ such that the boundaries of $f_{2}$ at $x_{1}$ and $x_{2}$ are both zero.
It is easy to verify that $f_{2}$ is a desired integer-valued $3$-flow.
\rule{3mm}{3mm}\par\medskip

\section{Proof of Theorem~\ref{TH: 3-to-4}}
\label{subsec: splitting}

First let us recall the vertex-splitting operation and Splitting Lemma.

\begin{definition}
Let $G$ be a graph and $v$ be a vertex.
If $F\subset E_{G}(v)$, we denote by $G_{(v;F)}$ the graph obtained from $G$ by splitting the edges of $F$ away from $v$.
That is, adding a new vertex $v^{*}$ and changing the common end of edges in $F$ from $v$ to $v^{*}$
(see Fig.~\ref{FIG: vertex splitting}).
\end{definition}

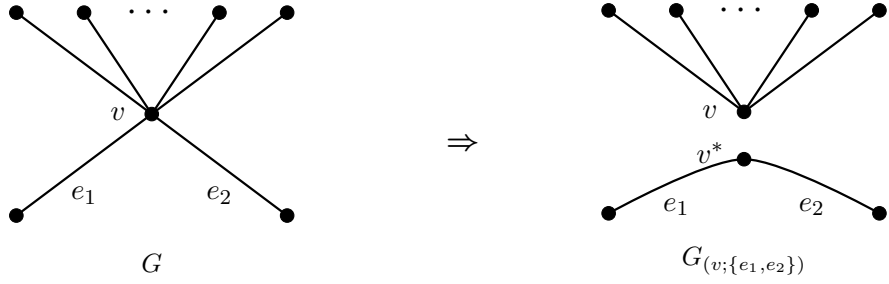
\begin{figure}
\begin{center}
\begin{minipage}[c]{0.48\textwidth}
\begin{center}
\begin{tikzpicture}[scale=0.9]
\path(-2,1.5) coordinate (u1);\draw [fill=black] (u1) circle (0.1cm);
\path(-1,1.5) coordinate (u2);\draw [fill=black] (u2) circle (0.1cm);
\path(1,1.5) coordinate (u3);\draw [fill=black] (u3) circle (0.1cm);
\path(2,1.5) coordinate (u4);\draw [fill=black] (u4) circle (0.1cm);
\path(0,0) coordinate (o);\draw [fill=black] (o) circle (0.1cm);
\path(-2,-1.5) coordinate (v1);\draw [fill=black] (v1) circle (0.1cm);
\path(2,-1.5) coordinate (v2);\draw [fill=black] (v2) circle (0.1cm);

\draw [line width=0.85] (u1)--(o)--(u2);
\draw [line width=0.85] (u3)--(o)--(u4);
\draw [line width=0.85] (v1)--(o)--(v2);
\node at (-0.5,0){$v$};
\node at (-1,-1.2){$e_1$};
\node at (1,-1.2){$e_2$};
\node at (0,1.5){\Large $\cdots$};
\node at (0,-2.2){\small $G$};
\end{tikzpicture}
\end{center}
\end{minipage}
\begin{minipage}[c]{0.02\textwidth}
\begin{center}
\begin{tikzpicture}[scale=0.9]
\node at (0,0){$\bm\Rightarrow$};
\end{tikzpicture}
\end{center}
\end{minipage}
\begin{minipage}[c]{0.48\textwidth}
\begin{center}
\begin{tikzpicture}[scale=0.9]
\path(-2,1.5) coordinate (u1);\draw [fill=black] (u1) circle (0.1cm);
\path(-1,1.5) coordinate (u2);\draw [fill=black] (u2) circle (0.1cm);
\path(1,1.5) coordinate (u3);\draw [fill=black] (u3) circle (0.1cm);
\path(2,1.5) coordinate (u4);\draw [fill=black] (u4) circle (0.1cm);
\path(0,0) coordinate (o);\draw [fill=black] (o) circle (0.1cm);
\path(0,-0.7) coordinate (o*);\draw [fill=black] (o*) circle (0.1cm);
\path(-2,-1.5) coordinate (v1);\draw [fill=black] (v1) circle (0.1cm);
\path(2,-1.5) coordinate (v2);\draw [fill=black] (v2) circle (0.1cm);

\draw [line width=0.85] (u1)--(o)--(u2);
\draw [line width=0.85] (u3)--(o)--(u4);
\draw [line width=0.85] plot [smooth] coordinates {(v1) (o*) (v2)};

\node at (-0.5,0){$v$};
\node at (-0.5,-0.6){$v^*$};
\node at (-1,-1.4){$e_1$};
\node at (1,-1.4){$e_2$};
\node at (0,1.5){\Large $\cdots$};
\node at (0,-2.2){\small $G_{(v;\{e_1,e_2\})}$};
\end{tikzpicture}
\end{center}
\end{minipage}
\end{center}
\caption{\small Splitting $\{e_1, e_2\}$ away from $v$}
\label{FIG: vertex splitting}
\end{figure}

\begin{lemma}
[Splitting Lemma~\cite{Fleischner1976,Fleischner1990}]
\label{LE: splitting}
Let $G$ be a bridgeless graph and $v$ be a vertex.
If $d_{G}(v)\geq 4$ and $e_{1}, e_{2},e_{3}\in E_{G}(v)$ are chosen in a way that $e_{1}$ and $e_{3}$
are in different blocks when $v$ is a cut-vertex,
then either $G_{(v;\{e_{1}, e_{2}\})}$ or $G_{(v;\{e_{1}, e_{3}\})}$ is bridgeless.
Furthermore, $G_{(v;\{e_{1}, e_{3}\})}$ is bridgeless if $v$ is a cut-vertex.
\end{lemma}

\medskip
\noindent{\bf Proof of Theorem~\ref{TH: 3-to-4}.}
Let $(G,\sigma)$ together with a flow $f_{1}$ be a counterexample to Theorem~\ref{TH: 3-to-4} such that
\begin{enumerate}
\item
$|\suppc(f_{1})|$ is minimized, where $\suppc(f_{1}) = E(G)-\supp(f_{1})$;
\item
subject to (1), $\sum_{v\in V(G)}|d_{G}(v)-3|$ is minimized.
\end{enumerate}

Now we use an argument similar to the one used in Section~\ref{SEC: 2-to-3}
and show that $(G,\sigma)$ actually admits an integer-valued $4$-flow
satisfying Theorem~\ref{TH: 3-to-4} in the following context.

\begin{claim}
\label{CL: nonempty}
$\supp(f_{1})\neq\emptyset$ and $\suppc(f_{1})\neq\emptyset$.
\end{claim}
\begin{proof}
If $\supp(f_{1}) =\emptyset$, then simply let $f_{2}(e) = 0$ for each edge $e$.
If $\suppc(f_{1}) =\emptyset$, then $\supp(f_{1}) = E(G)$ and thus $f_{1}$ itself is a modulo $3$-NZF of $(G,\sigma)$.
Since $G$ is bridgeless, Theorem~\ref{TH: Xu-Zhang-2} implies that $(G,\sigma)$ admits an integer-valued $3$-NZF $f_{2}$.
In both cases, $f_{2}$ is a desired integer-valued $4$-flow.
\end{proof}

\begin{claim}
\label{CL: max3}
The maximum degree of $G$ is at most $3$.
\end{claim}
\begin{proof}
Suppose that $G$ has a vertex $v$ with $d_{G}(v)\geq 4$.
Since $G$ is bridgeless, Lemma~\ref{LE: splitting} implies that
we can split a pair of edges $e_{1}, e_{2}$ from $v$ such that the resulting signed graph,
say $(G_{1},\sigma_{1})$, is still bridgeless.
In $G_{1}$, we consider $f_{1}$ as a mapping on $E(G_{1})$ and denote the common end of $e_{1}$ and $e_{2}$ by $v^{*}$.
Thus, $\partial f_{1}(v^{*})\equiv -\partial f_{1}(v)\pmod 3$.

Let $w\in\{v, v^{*}\}$.
If $\partial f_{1}(w)\equiv 0\pmod 3$ and $d_{G_{1}}(w) = 2$ with $E_{G_{1}}(w)=\{e_{w'}, e_{w''}\}$,
then we further suppress the vertex $w$ and denote the new edge by $e_{w}$ (see Fig.~\ref{FIG: G2}-(1)).
Then we can assign $e_{w}$ with value $f_{1}(e_{w'})$, signature $\sigma_{1}(e_{w'})\sigma_{1}(e_{w''})$,
and an orientation (based on its signature and value) in a way such that both ends of $e_{w}$ have zero boundary.
If $\partial f_{1}(w)\not\equiv 0\pmod 3$, then we further add a positive edge $vv^{*}$
oriented from $v$ to $v^{*}$ and assign $vv^{*}$ with value $\partial f_{1}(v^{*})$ (see Fig.~\ref{FIG: G2}-(2)).
In both cases, denote the resulting signed graph and mapping by $(G_{2},\sigma_{2})$ and $g_{1}$, respectively.

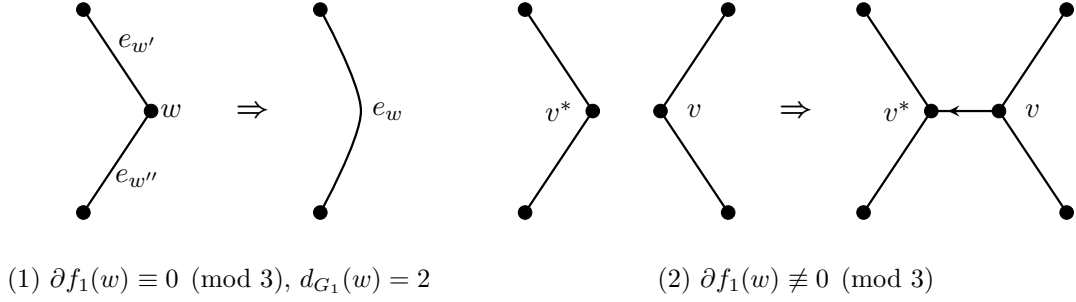
\begin{figure}
\begin{center}
\begin{minipage}[c]{0.44\textwidth}
\begin{center}
\begin{tikzpicture}[scale=0.9]
\path(-2.5,1.5) coordinate (u1);\draw [fill=black] (u1) circle (0.1cm);
\path(-2.5,-1.5) coordinate (u2);\draw [fill=black] (u2) circle (0.1cm);
\path(-1.5,0.0) coordinate (w);\draw [fill=black] (w) circle (0.1cm);
\path(0,0) coordinate (o);
\path(1,1.5) coordinate (v1);\draw [fill=black] (v1) circle (0.1cm);
\path(1,-1.5) coordinate (v2);\draw [fill=black] (v2) circle (0.1cm);
\path(1.6,0) coordinate (v3);
\draw [line width=0.85] (u1)--(w)--(u2);
\draw [line width=0.85] plot [smooth] coordinates {(v1) (v3) (v2)};
\node at (-1.2,0){$w$};
\node at (o){$\bm\Rightarrow$};
\node at (-1.7,1){$e_{w'}$};
\node at (-1.7,-1){$e_{w''}$};
\node at (2,0){$e_{w}$};
\node at (-0.5,-2.5){\small (1) $\partial f_{1}(w)\equiv 0\pmod 3$, $d_{G_{1}}(w) = 2$};
\end{tikzpicture}
\end{center}
\end{minipage}
\begin{minipage}[c]{0.55\textwidth}
\begin{center}
\begin{tikzpicture}[scale=0.9]
\path(-4, 1.5) coordinate (u1);\draw [fill=black] (u1) circle (0.1cm);
\path(-4,-1.5) coordinate (u2);\draw [fill=black] (u2) circle (0.1cm);
\path(-3,0) coordinate (u3);\draw [fill=black] (u3) circle (0.1cm);

\path(-1, 1.5) coordinate (v1);\draw [fill=black] (v1) circle (0.1cm);
\path(-1,-1.5) coordinate (v2);\draw [fill=black] (v2) circle (0.1cm);
\path(-2,0) coordinate (v3);\draw [fill=black] (v3) circle (0.1cm);

\path(1, 1.5) coordinate (y1);\draw [fill=black] (y1) circle (0.1cm);
\path(1,-1.5) coordinate (y2);\draw [fill=black] (y2) circle (0.1cm);
\path(2,0) coordinate (y3);\draw [fill=black] (y3) circle (0.1cm);

\path(4, 1.5) coordinate (x1);\draw [fill=black] (x1) circle (0.1cm);
\path(4,-1.5) coordinate (x2);\draw [fill=black] (x2) circle (0.1cm);
\path(3,0) coordinate (x3);\draw [fill=black] (x3) circle (0.1cm);

\path(0,0) coordinate (o);

\draw [line width=0.85] (u1)--(u3)--(u2);
\draw [line width=0.85] (v1)--(v3)--(v2);
\draw [line width=0.85] (x1)--(x3)--(x2);
\draw [line width=0.85] (y1)--(y3)--(y2);
\draw [directed, line width=0.85] (x3)--(y3);

\node at (-3.5,0){$v^*$};
\node at (-1.5,0){$v$};
\node at (o){$\bm\Rightarrow$};
\node at (3.5,0){$v$};
\node at (1.5,0){$v^*$};
\node at (0,-2.5){\small (2) $\partial f_{1}(w)\not\equiv 0\pmod 3$};
\end{tikzpicture}
\end{center}
\end{minipage}
\end{center}
\caption{\small Construction of signed graph $(G_{2},\sigma_{2})$}
\label{FIG: G2}
\end{figure}

It is easy to verify that $g_{1}$ is a modulo $3$-flow of $(G_{2},\sigma_{2})$
and $|\suppc(g_{1})|\leq |\suppc(f_{1})|$ and that
$\sum_{v\in V(G_{2})}|d_{G_{2}}(v)-3| <\sum_{v\in V(G)}|d_{G}(v)-3|$.
By the choice of $(G, \sigma)$, $(G_{2},\sigma_{2})$ has an integer-valued $4$-flow $g_{2}$
with $\supp(g_{1})\subseteq\{e\in E(G_{2})\colon g_{2}(e) =\pm 1,\pm 2\}$.
One can easily derive a desired integer-valued $4$-flow $f_{2}$ of $(G,\sigma)$ from $g_{2}$.
\end{proof}

Note that $G$ is connected. By Claim~\ref{CL: nonempty},
$G$ has a vertex $x$ such that $E_{G}(x)\cap\supp(f_{1})\neq\emptyset$ and $E_{G}(x)\cap\suppc(f_{1})\neq\emptyset$.
Let $e^{*}$ be an edge of $E_{G}(x)\cap\suppc(f_{1})$ and denote the other end of $e$ by $y$.
We may without lose of generality assume that $e^{*}$ is positive otherwise we make a switch at $x$.
We may further assume that $e^{*}$ is oriented from $x$ to $y$.
Now we contract $e^{*}$ and denote the resulting signed graph by $(G',\sigma')$.
Thus, the restriction of $f_{1}$ to $E(G')$, say $f_{1}'$, is a modulo $3$-flow of $(G',\sigma')$.
It follows from $\supp(f_{1}') =\supp(f_{1})$ that $|\suppc(f'_{1})| < |\suppc(f_{1})|$.
Hence, $(G',\sigma')$ admits an integer-valued $4$-flow $f'_{2}$
such that $\supp(f_{1}')\subseteq\{e\in E(G')\colon f'_{2}(e) =\pm1,\pm2\}$.

Now we consider the mapping $f'_{2}$ on $E(G)$.
Each vertex (possibly except $x$ and $y$) has zero boundary and $\partial f'_{2}(x) = -\partial f'_{2}(y)$.
If $\partial f'_{2}(x)\not\equiv 0\pmod 3$,
then we extend $f'_{2}$ to a mapping $h_{1}$ by assigning $h_{1}(e^{*})=-\partial f'_{2}(x)$.
Thus, $h_{1}$ is a modulo $3$-flow of $G$ with $\supp(h_{1})\supset\supp(f_{1})$.
This implies $|\suppc(h_{1})|<|\suppc(f_{1})|$, which contradicts the assumption~(1).
Thus, $\partial f'_{2}(x)\equiv 0\pmod 3$.
In summary, $x$ is a vertex satisfying $d_G(x)\leq3$, $E_{G}(x)\cap\suppc(f_{1})\neq\emptyset$, and
$1\leq |f'_{2}(e)|\leq 2$ for $e\in E_{G}(x)\cap\supp(f_{1})$.
Hence, $0\leq |\partial f'_{2}(x) |\leq 4$ and furthermore $|\partial f'_{2}(x)|\in\{0,3\}$.
Finally, we extend $f'_{2}$ to a mapping $f_{2}$ by assigning $f_{2}(e^{*}) = -\partial f'_{2}(x)$.
Clearly, $f_{2}$ is an integer-valued $4$-flow satisfying Theorem~\ref{TH: 3-to-4}.
\rule{3mm}{3mm}\par\medskip

\section{Proof of Theorem~\ref{TH: mod-circular}}

\subsection{A new vertex splitting lemma}

The vertex splitting method is one of the most useful techniques in graph theory
(especially, in the study of integer-valued flow and cycle cover problems).
In Section~\ref{subsec: splitting}, we have discussed Splitting Lemma
introduced by Fleischner (see Lemma~\ref{LE: splitting}).
Here are more early results about vertex splitting
by Nash-Williams~\cite{Nash-Williams1985}, Mader~\cite{Mader1978}, and Zhang~\cite{Zhang2002}.

\begin{theorem}
[\cite{Nash-Williams1985}]\label{TH: Nash-William}
Let $k$ be an even integer and $G$ be a $\lambda$-edge-connected graph.
Let $v\in V (G)$ and $a$ be an integer such that $\lambda\leq a$ and $\lambda\leq d(v)-a$.
Then there is an edge subset $F\subset E(v)$ such that $|F| = a$ and $G_{(v;F)}$ remains $\lambda$-edge-connected.
\end{theorem}

\begin{theorem}
[\cite{Mader1978}]
Let $G$ be a graph and $v\in V (G)$ such that $v$ is not a cut-vertex.
If $d(v)\geq4$ and $v$ is adjacent to at least two distinct vertices,
then there are two edges $e_{1},e_{2}\in E(v)$ such that,
for every pair of vertices $x, y\in V(G)-\{v\}$, the local edge-connectivity between $x$ and $y$
in the graph $G_{(v;\{e_{1},e_{2}\})}$ remains the same as in $G$.
\end{theorem}

\begin{theorem}
[\cite{Zhang2002}]
\label{TH: Zhang splitting HC}
Let $G$ be a graph with odd-edge-connectivity at least $\lambda_o$.
Let $v$ be a vertex of $G$ such that $d(v)\neq\lambda_o$ and $E(v) =\{e_0, e_1,\ldots, e_{d(v)-1}\}$.
Then there is a pair of edges $e_i, e_{i+1}\in E(v)$ (subindices modulo ${d(v)}$)
such that the graph $G_{(v;\{e_i,e_{i+1}\})}$ remains odd-$\lambda_o$-edge-connected.
\end{theorem}

\begin{definition}
Let $G$ be a graph and $v$ be a vertex.
Let $S(v)$ be a subset of $\{(e_{i},e_{j})\colon e_i, e_j\in E(v)~\text{and}~e_i\neq e_j\}$.
The subset $S(v)$ is {\em sequentially connected} if, for every pair of edges $e', e''\in E(v)$, there is a sequence
$(e_0,e_1), (e_1,e_2),\ldots, (e_{t-2},e_{t-1}), (e_{t-1},e_t)\in S(v)$ (subindices modulo $d(v)$)
such that $e'=e_0$ and $e''=e_t$.
\end{definition}

In Theorem~\ref{TH: Zhang splitting HC},
the subset $S(v) = \{(e_i, e_{i+1})\colon i\in\mathbb Z_{d(v)}\}$ is sequentially connected.
Therefore, the following theorem is a generalization of Theorem~\ref{TH: Zhang splitting HC},
and is expected to have many applications in graph theory.
The proof of Theorem~\ref{TH: splitting} is identical to the one in~\cite{Zhang2002}
and an alternative proof can be also found in~\cite{Sziget2008}.

\begin{theorem}
\label{TH: splitting}
Let $G$ be a graph with odd-edge-connectivity at least $\lambda_o$ and $v$ be a vertex with $d(v)\neq\lambda_o$.
Let $S(v)$ be a subset of $\{(e_i,e_j)\colon e_i, e_j\in E(v) ~\text{and}~ e_i\neq e_j\}$.
If the subset $S(v)$ is sequentially connected, then there is a pair of edges $(e',e'')\in S(v)$
such that the graph $G_{(v;\{e',e''\})}$ remains odd-$\lambda_o$-edge-connected.
\end{theorem}

The following corollary is an analog of Theorem~\ref{TH: Nash-William} with respect to odd-edge-connectivity.

\begin{corollary}
\label{CORO: splitting}
Let $G$ be a graph with odd-edge-connectivity at least $\lambda_o$ and $v$ be a vertex with $d(v) >\lambda_o$.
Let $S(v)=\{(e_i,e_j)\colon e_i, e_j\in E(v)~\text{and}~e_i\neq e_j\}$
and $a$ be an even integer such that $a\leq d(v)-\lambda_o$.
Then there is an edge subset $F\subset E(v)$ of size $a$, consisting of
disjoint elements of $S(v)$, such that $G_{(v;F)}$ remains odd-$\lambda_o$-edge-connected.
\end{corollary}

\begin{proof}
Let $a=2b$. Now we apply Theorem~\ref{TH: splitting} to $v$ repeatedly $b$ times at $v$.
Then the resulting graph remains odd-$\lambda_o$-edge-connected.
Denote by $\{v^*_1,\ldots, v^*_b\}$ the set of the resulting vertices of degree two.
It is easy to see that the collection of the edges incident with $v^*_i$ 
for $i=1,\ldots,b$ is a desired edge subset $F$ of $E(v)$.
\end{proof}

\subsection{An application of Tutte's \texorpdfstring{$f$}{f}-factor theorem}

Theorem~\ref{TH: mod-circular} will be proved by applying both
Theorem~\ref{TH: splitting} and some $f$-factor lemmas (such as, Lemma~\ref{LE: mu-p-factor})
in this section.

\begin{definition}
Let $G$ be a graph and $f\colon V(G) \to \mathbb Z^+$ be a mapping.
An {\em $f$-factor} of $G$ is a subgraph $H$ such that
$d_{H}(v)= f(v)$ for each vertex $v\in V(G)$. In particular,
if the range of $f$ is $\{1,2\}$, we simply call $H$ a {\em $\{1,2\}$-factor}.
\end{definition}

In~\cite{Tutte1952}, Tutte gave a necessary and sufficient condition of the existence of $f$-factors.

\begin{theorem}
[\cite{Tutte1952}]
\label{LE: tutte-factor}
A graph $G$ has an $f$-factor if and only if for any two disjoint vertex subsets $S,T \subseteq V(G)$,
\begin{equation}
\sum_{v\in S}f(v)\geq |\mathcal{O}(S,T)| +\sum_{v\in T} [f(v)-d_{G-S}(v)],
\label{EQ: tutte-factor-1}
\end{equation}
where $\mathcal{O}(S,T)$ is the set of components $U$ of $G - S - T$ for which
\begin{equation}
\sum_{v\in U}f(v)+e(U,T)\equiv 1\pmod 2.
\label{EQ: odd-component}
\end{equation}
\end{theorem}

Next we apply Tutte's $f$-factor theorem to find a $\{1,2\}$-factor for graphs defined below.

\begin{lemma}\label{LE: 1-2-factor}
Let $k$ be an odd integer and $G$ be an odd-$k$-edge-connected graph.
Let $\{V_1, V_2\}$ be a partition of $V(G)$ such that $d_{G}(v) = k$ if $v\in V_1$ and $d_{G}(v)=2k$ if $v\in V_2$.
If $f$ is a function satisfying $f(v)={d_{G}(v)}/{k}$ for each vertex $v$, then $G$ has an $f$-factor.
\end{lemma}

\begin{Proof}
Let $S$ and $T$ be two disjoint subsets of $V(G)$ and $\mathcal{O} = \mathcal{O}(S,T)$.
Let $\{Q_1,Q_2,Q_3,Q_4\}$ be a partition of $T$,
where for each $t\in\{1,2\}$, $Q_t$ consists of the vertices $v \in T\cap V_t$ such that $d_{G-S}(v)=0$,
$Q_3$ consists of the vertices $v$ of $T \cap V_2$ such that $d_{G-S}(v)=1$, and $Q_4=T-Q_1-Q_2-Q_3$.
The following claim directly follows from the definitions.

\begin{claim}
\begin{enumerate}
\item $kf(v) = d_G(v)$ and $f(v) \equiv d_G(v) \pmod 2$ for each vertex $v$.
\item $\sum_{v\in U}d_{G}(v)+e(U,T)\equiv 1\pmod 2$ for each $U \in \mathcal{O}$.
\end{enumerate}
\label{basic}
\end{claim}

We partition $\mathcal{O}$ into $\mathcal{O}_1$ and $\mathcal{O}_2$, where

\begin{center}
$\mathcal{O}_1 = \{U \in \mathcal{O} \colon e(U,T) = 0\}$\quad\text{and}\quad
$\mathcal{O}_2 = \{U \in \mathcal{O} \colon e(U,T) \not = 0\}$.
\end{center}

\begin{claim}
\label{U-S}
 $$\sum_{U \in \mathcal{O}} e(U,S) \geq k |\mathcal{O}_1| + |\mathcal{O}_2|.$$
\end{claim}
\begin{proof}
Note that if $U \in \mathcal{O}_1$, then $e(U,T) = 0$ and thus $E(U,S)$ is an edge-cut.
Since $G$ is odd-$k$-edge-connected, it suffices to show that for each $U \in \mathcal{O}$, $e(U,S) \equiv 1 \pmod 2.$

For each $U \in \mathcal{O}$, we have
\[\sum_{v\in U}d_{G}(v) \equiv e(U,T) + e(U,S) \equiv -e(U,T) + e(U,S) \pmod 2.\]
Thus by Claim~\ref{basic}-(2), we have $e(U,S) \equiv 1 \pmod 2$.
\end{proof}

\begin{claim}
\label{ST}
\[e(S,T)=\sum_{v\in T}[d_G(v)-d_{G-S}(v)]\geq k\sum_{v\in T}[f(v)-d_{G-S}(v)]+(k-1)|\mathcal{O}_2|.\]
\end{claim}

\begin{proof}
Since $d_{G-S}(v) = 0$ if $v \in Q_1\cup Q_2$ and $d_{G-S}(v) = 1$ if $v \in Q_3$, we have
\begin{equation}
\label{T1}
\sum_{v\in Q_1\cup Q_2\cup Q_3}[d_G(v)-d_{G-S}(v)]=k\sum_{v\in Q_1\cup Q_2\cup Q_3}[f(v)-d_{G-S}(v)]+(k-1)\sum_{v\in Q_3}d_{G-S}(v).
\end{equation}

Since $kf(v) = d_G(v)$ for each vertex $v$, we have
\begin{equation}
\label{TQ4}
\sum_{v\in Q_4} [d_G(v) - d_{G-S}(v)] = \sum_{v\in Q_4} [kf(v) - d_{G-S}(v)]
= k\sum_{v\in Q_4} [f(v) - d_{G-S}(v)] + (k-1)\sum_{v\in Q_4} d_{G-S}(v).
\end{equation}

Combining (\ref{T1}) and (\ref{TQ4}), we have
\begin{equation}
\label{TQ5}
\sum_{v\in T} [d_G(v) - d_{G-S}(v)] = k \sum_{v\in T} [f(v) - d_{G-S}(v)] + (k-1) \sum_{v\in Q_3\cup Q_4} d_{G-S}(v).
\end{equation}

Since each vertex $v \in Q_3 \cup Q_4$ is adjacent to at most $d_{G-S}(v)$ components in $\mathcal{O}_2$, we have
\begin{equation}
\label{Q-O}
\sum_{v\in Q_3\cup Q_4} d_{G-S}(v) \geq |\mathcal{O}_2|.
\end{equation}

Combining (\ref{TQ5}) and (\ref{Q-O}), we have
\[e(S,T) \geq k \sum_{v\in T} [f(v) - d_{G-S}(v)] + (k-1)|\mathcal{O}_2|.\]
\end{proof}

Denote $S^{\mathrm{c}}=V(G)-S$. Now we are to estimate $e(S,S^{\mathrm{c}})$
in two ways by finding a lower bound and an upper bound. Obviously,
\begin{equation}
\label{Supper}
e(S,S^c) \leq \sum_{v\in S} d_G(v) = k \sum_{v\in S} f(v).
\end{equation}

On the other hand,
\begin{equation}
\label{Slower}
e(S,S^c)
\geq e(S,T) + \sum_{U\in \mathcal{O}} e(S,U).
\end{equation}

By (\ref{Supper}) and (\ref{Slower}) together with Claims~\ref{U-S} and \ref{ST}, we have
\begin{equation}
\label{final}
\begin{split}
k \sum_{v\in S} f(v)
& \geq k \sum_{v\in T} [f(v) - d_{G-S}(v)] + (k-1)|\mathcal{O}_2| + k |\mathcal{O}_1| + |\mathcal{O}_2|\\
& = k \sum_{v\in T} [f(v) - d_{G-S}(v)] + k(|\mathcal{O}_1|+ |\mathcal{O}_2|) \\
& = k \left(\sum_{v\in T} [f(v) - d_{G-S}(v)] + |\mathcal{O}|\right).
\end{split}
\end{equation}

By (\ref{final}), we have
\[\sum_{v\in S} f(v) \geq |\mathcal{O}| + \sum_{v\in T} [f(v) - d_{G-S}(v)]. \]
Therefore, by Theorem~\ref{LE: tutte-factor}, $G$ has an $f$-factor.
\end{Proof}

\begin{lemma}
\label{LE: mu-p-factor}
Let $G$ be a graph with odd-edge-connectivity at least $(2p+1)$.
If there is a mapping $\mu: V(G)\rightarrow\mathbb Z^{+}$ such that
$d_{G}(v) = (2p+1)\mu(v)$ for each vertex $v\in V(G)$,
then there is a spanning subgraph $F$ such that $d_F(v) = p\mu(v)$.
\end{lemma}

\begin{Proof}
For each vertex $v$ with $d_{G}(v)\notin\{2p+1, 2(2p+1)\}$,
we first apply Corollary~\ref{CORO: splitting} to $v$ with $a = 2(2p+1)$ and $\lambda_o=2p+1$.
Repeatedly apply this process until the degree of every vertex is either $(2p+1)$ or $2(2p+1)$.
Let $G'$ denote the resulting graph.

Next we apply Lemma~\ref{LE: 1-2-factor} to $G'$ with $k=2p+1$.
Let $F_0$ be a $\{1,2\}$-factor of $G'$ such that, for each $v\in V(G')$,
$d_{F_0}(v)=1$ if $d_{G'}(v)=2p+1$ and $d_{F_0}(v)=2$ if $d_{G'}(v)=2(2p+1)$.

Let $G''=G'-E(F_0)$. Split each vertex $v$ of $G''$ with $d_{G''}(v)=4p$
into a pair of degree $2p$ vertices (no need to preserve the odd-edge-connectivity here).
Let $G'''$ be the resulting $2p$-regular graph.
By Petersen's Theorem, $G'''$ has a $2$-factorization $\{F_1,\ldots, F_p\}$.

When $p$ is even, say $p=2q$, the subgraph $F$ induced by the edges of $F_1,\ldots, F_q$ is a desired spanning subgraph.
When $p$ is odd, say $p=2q+1$, the subgraph $F$ induced by the edges of $F_0, F_1,\ldots, F_q$ is a desired spanning subgraph.
\end{Proof}

\subsection{Completion of the proof of Theorem~\ref{TH: mod-circular} }

Now we are ready to complete the proof of Theorem~\ref{TH: mod-circular}.

It is obvious that {\bf(C)} implies {\bf(B)}.
Since {\bf(A)} and {\bf(B)} in Theorem~\ref{TH: Jaeger} are equivalent,
we will prove that {\bf(A)} implies {\bf(C)}.

Let $(G,\sigma)$ be an odd-$(2p+1)$-edge-connected signed graph and
$\tau$ be a modulo $(2p+1)$-orientation of $(G,\sigma)$.
We are going to show that $(G,\sigma)$ has an integer-valued circular $(2+\frac{1}{p})$-flow.

For each $v\in V(G)$, denote $H_{\tau}^{+}(v)=\{h \in H(v)\colon \tau(v)=1\}$
and $H_{\tau}^{-}(v)=\{h \in H(v)\colon \tau(v)=-1\}$.
Let $d_{\tau}^{+}(v) = |H_{\tau}^{+}(v)|$ and $d_{\tau}^{-}(v) = |H_{\tau}^{-}(v)|$.
If both $d_{\tau}^{+}(v) >0$ and $d_{\tau}^{-}(v) > 0$ for some vertex $v$,
then by Theorem~\ref{TH: splitting} with
$S(v) = \{(e',e'') \colon e'\in H_{\tau}^{+}(v)~\text{and}~e''\in H_{\tau}^{-}(v)\}$,
one can split a pair of half-edges
(one from $H_{\tau}^{+}(v)$ and the other from $H_{\tau}^{-}(v)$) away from $v$
and then suppress the resulting degree $2$ vertex.
Let $G'$ be the resulting graph obtained from $G$ by repeatedly
applying Theorem~\ref{TH: splitting} until no such pair of edges exits.
Then $G'$ remains odd-$(2p+1)$-edge-connected.
Since $\tau$ remains a modulo $(2p+1)$-orientation of $(G',\sigma)$ and
either $d_{\tau}^+(v) =0$ or $d_{\tau}^-(v)=0$ for each vertex $v$ of $G'$,
there is a mapping $\mu$ of $G'\colon V(G')\rightarrow\mathbb Z^+$ such that $d_{G'}(v) = (2p+1)\mu(v)$.

By Lemma~\ref{LE: mu-p-factor}, $G'$ has a spanning subgraph $F$ such that
$d_F(v)=p\mu(v)$. Then the integer-valued function $f^*$ defined as follows
is a circular $(2+\frac{1}{p})$-flow of $(G, \sigma)$:
\[ f^*(e) =\left\{\begin{array}{ll}
 p &\mbox{if $e\not\in F$};\\
 -p-1 &\mbox{if $e\in F$}.\end{array}\right.\]
\rule{3mm}{3mm}\par\medskip


\end{document}